\documentclass[12pt]{amsart}

\usepackage{amssymb,amsthm}

\newtheorem{theorem}{Theorem}[section]
\newtheorem{corollary}[theorem]{Corollary}

\newtheorem{lemma}[theorem]{Lemma}

\begin{document}

\title[Maximal abelian subgroups]{A lower bound on the size of maximal abelian subgroups}

\author[M. L. Lewis]{Mark L.\ Lewis}
\address{Department of Mathematical Sciences, Kent State University, Kent, OH 44242}
\email{lewis@math.kent.edu}


\begin{abstract}
Let $G$ be a $p$-group for some prime $p$.  Let $n$ be the positive integer so that $|G:Z(G)| = p^n$.  Suppose $A$ is a maximal abelian subgroup of $G$.  Let $$p^l = {\rm max} \{|Z(C_G (g)):Z(G)| : g \in G \setminus Z(G)\},$$ $$p^b = {\rm max} \{|cl(g)| : g \in G \setminus Z(G) \},$$ and $p^a = |A:Z(G)|$.  Then we show that $a \ge n/(b+l)$. 
\end{abstract}

\maketitle

\section{Introduction}

Throughout this paper, all groups are finite.  Suppose $G$ is a group and $A$ is a maximal abelian subgroup of a $p$-group $G$ for some prime $p$.  In this paper, we want to obtain a lower bound on the size of $|A|$.  We will see that it is not difficult to show that $Z(G) \le A$, and so, in fact, we will find a lower bound on $|A:Z(G)|$.   

We are going to bound $|A: Z(G)|$ in terms of two quantities.  The first quantity is going to be maximum of $\left\{ |G:C_G(x)| : x \in G \setminus Z (G) \right\}$.  This, of course, is the maximum size of a conjugacy class of $G$.  The second quantity is related.  Given an element $a \in G \setminus Z(G)$, we define the {\it center} of the element $a$ to be $Z(a) = Z (C_G (a))$.  We define ${\mathcal Z} (G) = \left\{ Z (a) : a \in G \setminus Z (G) \right\}$.  The second quantity we use is the maximum of $\left\{ |Z:Z(G)|: Z \in {\mathcal Z} (G) \right\}$.   With these two quantities in hand, we will obtain the following lower bound on the index of $|A:Z(G)|$ in terms of these quantities.

\begin{theorem} \label{main}
Let $G$ be a $p$-group with $|G:Z(G)| = p^n$.  Suppose $A$ to be a maximal abelian subgroup of $G$.  Set 
$$p^l = {\rm max} \{|Z:Z(G)| : Z \in {\mathcal Z} (G)\}$$ and 
$$p^b = {\rm max} \{|G:C_G(x)| : x \in G \setminus Z(G)\}.$$ 

If $p^a = |A:Z(G)|$, then $a \ge n/(b+l)$.
\end{theorem}

Theorem \ref{main} gives a global lower bound for all of the maximal abelian subgroups of $G$.  However, if we specialize our quantities to a specific maximal abelian subgroup $A$ of $G$, we can obtain a more precise lower bound for $|A:Z (G)|$.  We replace the first quantity by the maximum of $\left\{ |G:C_G(x)| : x \in A \setminus Z (G) \right\}$.  This, of course, is the maximum size of a conjugacy class that intersects $A$.  For the second quantity we set ${\mathcal Z} (G \mid A) = \{ Z \in {\mathcal Z} (G) : Z \le A \}$, and observe that ${\mathcal Z} (G \mid A) = \{ Z(a) \mid a \in A \setminus Z(G)\}$.  We then replace the second quantity by the maximum of $\left\{ |Z:Z(G)|: Z \in {\mathcal Z} (G \mid A) \right\}$.   With these two quantities in hand, we will obtain the following lower bound on the index of $|A:Z(G)|$ in terms of these quantities.

\begin{theorem} \label{threeia}
Let $G$ be a $p$-group with $|G:Z(G)| = p^n$.  Suppose $A$ to be a maximal abelian subgroup of $G$.  Set 
$$p^l = {\rm max} \{|Z:Z(G)| : Z \in {\mathcal Z} (G \mid A)\}$$ and 
$$p^b = {\rm max} \{|G:C_G(x)| : x \in A \setminus Z(G)\}.$$ 

If $p^a = |A:Z(G)|$, then $a \ge n/(b+l)$.
\end{theorem}

Notice that Theorem \ref{main} is an immediate consequence of Theorem \ref{threeia}, and hence, we will prove Theorem \ref{threeia}.

We believe that $Z(g)$ should be studied in more depth and we will prove several results regarding $Z(g)$ in more general contexts than $p$-groups.  In fact, most of our preliminary results are proved for all groups.  Only when we go to obtain our bound do we restrict ourselves to $p$-groups.  We initially proved this result in the setting of semi-extraspecial $p$-groups, and in Section \ref{ses}, we will specialize our results in that case.  We realized that our techniques worked in general groups and that our result could be proved for $p$-groups, and so, we have decided to present our work in the general setting and result in the setting of $p$-groups.   The initial question that motivated this work was asked by Thomas Keller and we would like to thank Thomas for his helpful conversations that led to this paper.  We would also like to thank Rachel Camina and Sophia Brenner for a careful reading of this paper and helpful comments and Rolf Brandl for several useful comments.  Finally, we would like to thank Josh Maglione for the computations that inform the examples in Section \ref{ses}.

\section{The Centers of Centralizers}

In this section, we prove some results regarding $C_G (a)$ and $Z(a)$.  This first result is well known.

\begin{lemma} \label{well}
Let $G$ be a group.  Then $Z(G) = \cap_{x \in G} C_G (x)$.
\end{lemma}

Given $g \in G \setminus Z(G)$, we write $Z_G (g) = Z (C_G (g))$.  If it is clear what $G$ is, then we will suppress the $G$ and write $Z(g) = Z_G (g)$.

\begin{lemma} \label{threee}
Let $G$ be a group.  If $a \in G \setminus Z(G)$, then 
$$Z (a) = \cap_{b \in C_G (a)} C_G (b).$$
\end{lemma}

\begin{proof}
By definition, we have $Z (a) = Z (C_G (a))$.  Applying Lemma \ref{well}, we have
$$Z (C_G (a)) = \cap_{b \in C_G (a)} C_{C_G (a)} (b) = \cap_{b \in C_G (a)} (C_G (a) \cap C_G(b)).$$
Since $a \in C_G (a)$, we see that 
$$\cap_{b \in C_G (a)} (C_G (a) \cap C_G(b)) = C_G (a) \cap (\cap_{b \in C_G (a)} C_G (b)) = \cap_{b \in C_G (a)} C_G (b),$$ and this proves the lemma.
\end{proof}

If $G$ is any group, we can set ${\mathcal C} (G) = \{ C_G (a) \mid a \in G \setminus Z(G) \}$ and ${\mathcal Z} (G) =\{ Z(a) \mid a \in G \setminus Z(G) \}$.  





We next show that if $a,b$ are not central elements of $G$ then $a \in C_G (b)$ if and only if $Z(a) \le C_G (b)$ and this occurs if and only if $Z(b) \le C_G (a)$.

\begin{lemma} \label{three}
Let $G$ be a group and suppose $a,b \in G \setminus Z(G)$.
\begin{enumerate}
\item $a \in C_G(b)$ if and only if $Z(a) \le C_G (b)$.
\item $Z(a) \le C_G (b)$ if and only if $Z(b) \le C_G (a)$.
\end{enumerate}
\end{lemma}

\begin{proof}
Suppose $a \in C_G (b)$.  This implies that $a$ centralizes $b$, and so, $b$ centralizes $a$.  This implies that $b \in C_G (a)$.  We see that $b$ centralizes $Z(a) = Z(C_G(a))$.  Finally, $Z(a)$ centralizes $b$, and so, $Z(a) \le C_G (b)$. Now, suppose $Z(a) \le C_G (b)$.  This implies that $a \in C_G (b)$.  This proves (1).  If $Z(a) \le C_G (b)$, $a$ centralizes $b$.  Now, $b \in C_G (a)$, and so, $Z(b) \le C_G (a)$ by (1).  The converse of (2) follows similarly.
\end{proof}

We now obtain a characterization of abelian subgroups in terms of the centers and centralizers of their elements.

\begin{lemma} \label{threef}
Let $G$ be a group and let $H$ be a subgroup of $G$.  Then $H$ is abelian if and only if $Z(a) \le C_G (b)$ for all $a,b \in H \setminus Z(G)$.
\end{lemma}

\begin{proof}
Suppose $H$ is abelian, and suppose that $a,b \in H \setminus Z(G)$.  Then $H \le C_G (b)$.  Since $a \in H$, this implies $a \in C_G (b)$.  By Lemma \ref{three}, we have that $Z(a) \le C_G (b)$.  Conversely, suppose for all $a,b \in H \setminus Z(G)$ that $Z(a) \le C_G (b)$.  This implies that $a \in C_G (b)$, and so, $a$ centralizes $b$.  Thus, $ab=ba$ for all $a,b \in H \setminus Z(G)$.  This implies that $H$ is abelian.
\end{proof}

\section{Maximal abelian subgroups}

We now prove the results regarding maximal abelian subgroups.  We first see that maximal abelian subgroup contain the centers of their elements.

\begin{lemma} \label{threeha}
Let $G$ be a group, and let $A$ be a maximal abelian subgroup of $G$.  If $a \in A \setminus Z(G)$, then $Z(a) \le A$.  In particular, if $A$ is maximal abelian, then $Z(G) \le A$.
\end{lemma}

\begin{proof}
Since $a \in A$ and $A$ is abelian, we have $A \le C_G (a)$.  By definition, $Z(a)$ is the center of $C_G(a)$, and so $Z(a)$ centralizes $A$.  Since $A$ and $Z(a)$ are abelian, we see that $AZ(a)$ will be abelian.  Finally, the fact that $A$ is maximal abelian implies $AZ(a) = A$, and hence, $Z(a) \le A$.  Since $Z(G) \le Z(a)$, this implies $Z (G) \le A$.
\end{proof}

The follow observation is well known. 

\begin{lemma}
If $G$ is a group and $C_G (g)$ is abelian, then $C_G (g)$ is maximal abelian.
\end{lemma}

\begin{proof}
Let $A$ be an abelian subgroup containing $C_G (g)$.  Since $g \in A$, we see that $A \subseteq C_G (g)$, and so $A = C_G (g)$.  It follows that $C_G (g)$ is maximal abelian.	
\end{proof}

We now find two sufficient conditions for a maximal abelian subgroup to be the centralizer of an element in $G$.

\begin{lemma} \label{cyclic}
Let $G$ be a group and assume $A$ is maximal abelian in $G$.  
\begin{enumerate}
\item If $A/Z(G) = \langle a, Z(G) \rangle$ for some element $a \in A$, then $A = C_G (a)$.
\item If $|A:Z (G)|$ is a prime, then $A = C_G (a)$ for some element $a \in A$.
\end{enumerate}  
\end{lemma}

\begin{proof}
Suppose $A = \langle a, Z(G) \rangle$ for some element $a \in A$.  Clearly, $A \le C_G (a)$.  Suppose $b \in C_G (a)$.  Then $b$ centralizes $a$ and $Z(G)$, so $b$ centralizes $A$.  This implies that $\langle A, b \rangle$ is abelian.  Since $A$ is maximal abelian, this implies that $b \in A$.  We conclude that $A = C_G (a)$.   Suppose that $|A:Z(G)|$ is a prime.  Then $A = \langle a, Z(G) \rangle$ for some element $a \in A$.  We then apply (1).
\end{proof}




We now consider a set of elements in $G$.  We see if the set lies in the intersection of the centralizers of the elements in the set, then the product of the centers lie in the center of the intersection of the centralizers.

\begin{lemma}\label{threeg1}
Let $G$ be a group and let $a_1, \dots, a_n$ be elements in $G \setminus Z(G)$.	 If $a_1, \dots, a_n \in \cap_{i=1}^n C_G (a_i)$, then $\prod_{i=1}^n Z(a_i) \le Z (\cap_{i=1}^n C_G (a_i))$.
\end{lemma}

\begin{proof}
Suppose that 
$$a_1, \dots, a_n \in \cap_{i=1}^n C_G (a_i).$$  Observe that $a_j \in C_G (a_i)$ for each $1 \le i,j, \le n$.  Applying Lemma \ref{three} (1), we have $Z(a_j) \le C_G(a_i)$ for each $i$ and $j$.  This implies that 
$$\prod_{j=1}^n Z(a_j) \le C_G (a_i)$$ for each $i$, and so, 
$$\prod_{j=1}^n Z(a_j) \le \cap_{i=1}^n C_G (a_i).$$  Note that if $b \in \cap_{i=1}^n C_G (a_i)$, then $b \in C_G (a_i)$ for each $i$, and so, $b$ centralizes $Z(a_j)$ for each $j$.  It follows that $\cap_{i=1}^n C_G (a_i)$ centralizes each $Z (a_j)$.  This proves that 
$$\prod_{j=1}^n Z(a_j) \le Z (\cap_{i=1}^n C_G (a_i)).$$ 	
\end{proof}

Using Lemma \ref{threeg1}, we can determine when the product of centers of elements will be abelian.  The next two results are generalizations of results that have been proven by Konieczny. This next lemma is a generalization of Proposition 1.2 (1) of \cite{kon}.     

Note that it is not clear in general that $\prod_{i=1}^n Z(a_i)$ will be a subgroup.  In fact, it is probably easy to find examples where it is not a subgroup.  Thus, we need to add this as an assumption in the next two results.  However, when the $a_i$'s commute, it is not difficult to see that $\prod_{i=1}^n Z(a_i)$ will be a subgroup which is why we do not need to include this hypothesis in Lemma \ref{threeg1} or Theorem \ref{threeg5}.

\begin{lemma} \label{threeg2}
Let $G$ be a group and let $a_1, \dots, a_n$ be elements in $G \setminus Z(G)$.   Assume $\prod_{i=1}^n Z(a_i)$ is a subgroup of $G$.  The subgroup $\prod_{i=1}^n Z(a_i)$ is abelian if and only if $\prod_{i=1}^n Z(a_i) \le \cap_{i=1}^n C_G (a_i).$
\end{lemma}

\begin{proof}
Now, suppose that 
$$\prod_{i=1}^n Z(a_i) \le \cap_{i=1}^n C_G (a_i).$$  This implies that 
$$a_1, \dots, a_n \in \cap_{i=1}^n C_G (a_i),$$ and so, by Lemma \ref{threeg1}, $\prod_{i=1}^n Z(a_i)$ is abelian.  

Conversely, suppose that $\prod_{i=1}^n Z(a_i)$ is abelian.  For each $j$, observe that the element $a_j \in \prod_{i=1}^n Z(a_i)$, and since $\prod_{i=1}^n Z(a_i)$ is abelian, we have 
$$\prod_{i=1}^n Z(a_i) \le C_G (a_j).$$  This implies that $a_i \in C_G (a_j)$ for each $i$ and $j$, and so, 
$$a_i \in \cap_{j=1}^n C_G (a_j).$$  We now have 
$$a_1, \dots, a_n \in \cap_{i=1}^n C_G (a_i),$$ and by Lemma \ref{threeg1}, we deduce that 
$$\prod_{i=1}^n Z(a_i) \le \cap_{i=1}^n C_G (a_i).$$  
\end{proof}

Further, we determine that a product of centers of elements will be maximal abelian when it is the intersection of the centralizers of the elements.  We should note that partial results along these lines appear as Proposition 1.2 (2) and (3) of \cite{kon}.

\begin{lemma} \label{threeg3}
Let $G$ be a group and let $a_1, \dots, a_n$ be elements in $G \setminus Z(G)$. Assume $\prod_{i=1}^n Z(a_i)$ is a subgroup of $G$.  The subgroup $\prod_{i=1}^n Z(a_i)$ is maximal abelian if and only if $$\prod_{i=1}^n Z(a_i) = \cap_{i=1}^n C_G (a_i).$$
\end{lemma}

\begin{proof}
Suppose that 
$$\prod_{i=1}^n Z(a_i) = \cap_{i=1}^n C_G (a_i).$$  By Lemma \ref{threeg2}, we see that $\prod_{i=1}^n Z(a_i)$ is abelian.  By way of contradiction, assume $\prod_{i=1}^n Z(a_i)$ is not maximal abelian.  Then there exists an abelian subgroup $A$ so that $\prod_{i=1}^n Z(a_i) < A$.  Notice that $a_i \in A$ for each $i$, so $A \le C_G (a_i)$.  This implies that 
$$\prod_{i=1}^n Z(a_i) < A \le \cap_{i=1}^n C_G (a_i)$$ which contradicts
$$\prod_{i=1}^n Z(a_i) = \cap_{i=1}^n C_G (a_i).$$  Therefore, $\prod_{i=1}^n Z(a_i)$ is maximal abelian.

Conversely, suppose $\prod_{i=1}^n Z(a_i)$ is maximal abelian.  By Lemma \ref{threeg2}, we have 
$$\prod_{i=1}^n Z(a_i) \le \cap_{i=1}^n C_G (a_i)$$ and then by Lemma \ref{threeg1}, 
$$\prod_{i=1}^n Z(a_i) \le Z (\cap_{i=1}^n C_G (a_i)).$$  If 
$$\prod_{i=1}^n Z(a_i) < \cap_{i=1}^n C_G (a_i),$$ then there exists an element 
$$a \in \cap_{i=1}^n C_G (a_i) \setminus \prod_{i=1}^n Z(a_i).$$  Notice that 
$$\prod_{i=1}^n Z(a_i) < \langle a, \prod_{i=1}^n Z(a_i) \rangle$$ and 
$$\langle a, \prod_{i=1}^n Z(a_i) \rangle$$ is an abelian group.  This contradicts the maximality of $\prod_{i=1}^n Z(a_i)$.  Thus, we conclude $$\prod_{i=1}^n Z(a_i) = \cap_{i=1}^n C_G (a_i)$$.
\end{proof}

We now obtain equivalent characterizations of maximal abelian subgroups.

\begin{theorem}\label{threeg5}
Let $G$ be a group and let $A$ be a subgroup of $G$.  Then the following are equivalent:
\begin{enumerate}
	\item $A$ is maximal abelian.
	\item $$A = \prod_{a \in A \setminus Z(G)} Z(a) = \cap_{a \in A \setminus Z(G)} C_G (a).$$
	\item There is a subset $a_1, \dots, a_n$ in $A \setminus Z(G)$ so that $$A = \prod_{i=1}^n Z(a_i) = \cap_{i=1}^n C_G (a_i).$$
\end{enumerate}
\end{theorem}

\begin{proof}
Let $A$ be a maximal abelian subgroup of $G$.  If $b \in A \setminus Z(G)$, then $b \in Z(b) \cap A \setminus Z(G)$, and by Lemma \ref{threeha}, we have $Z(b) \le A$. 
This implies that $\prod_{a \in A \setminus Z(G)} Z(a) \le A$ and since $Z(G) \le Z(a)$ for all $a$, we see that $A \le \prod_{a \in A \setminus Z(G)} Z(a)$.  Hence, $A = \prod_{a \in A \setminus Z(G)} Z(a)$.	Since $\prod_{a \in A \setminus Z(G)} Z(a)$ is maximal abelian, we may apply Lemma \ref{threeg3} to see that $\prod_{a \in A \setminus Z(G)} Z(a) = \cap_{a \in A \setminus Z(G)} C_G (a)$.  This proves that (1) implies (2).  We claim that (2) implies (3) is obvious.  Now, suppose there is a subset $a_1, \dots, a_n$ in $A \setminus Z(G)$ so that 
$$A = \prod_{i=1}^n Z(a_i) = \cap_{i=1}^n C_G (a_i).$$  
Applying Lemma \ref{threeg3}, we see that $A$ is maximal abelian.
\end{proof}

We now specialize our results to $p$-groups.   In this next theorem, we obtain a lower bound for the size of any maximal abelian subgroup of a $p$-group.  We are now ready to prove Theorem \ref{threeia} from the Introduction.


\begin{proof}[Proof of Theorem \ref{threeia}]
	
Applying Theorem \ref{threeg5}, $A = \prod_{x \in A \setminus G'} Z(x)$.  We can find elements 
$$g_1, g_2, \dots, g_t \in A \setminus Z(G)$$ so that $A = \prod_{i=1}^t Z (g_i) = \cap_{i=1}^t C_G (g_i)$ and $g_i \not\in \prod_{j=1}^{i-1} Z(g_j)$ for each $i$ with $2 \le i \le t$.    We obtain 
$$|G:A| = |G:\cap_{i=1}^t C_G (g_i)| \le \prod_{i=1}^t |G:C_G (g_i)| \le (p^b)^t = p^{bt}$$ and 
$$|A:Z(G)| = |\prod_{i=1}^t Z(g_i):Z(G)| \le \prod_{i=1}^t |Z(g_i):Z(G)| \le (p^l)^t = p^{lt}.$$  
It follows that 
$$p^n = |G:Z(G)| = |G:A||A:Z(G)| \le p^{bt} p^{lt} = p^{bt + lt} = p^{(b+l)t}.$$  This yields $n \le (b+l)t$, and so, $t \ge n/(b+l)$.  Now, we see that 
$$|A:Z(G)| = |A:\prod_{i=1}^{t-1} Z(g_i)|\cdots |Z(g_1)Z(g_2):Z(g_1)| |Z(g_1):Z(G)|.$$  
Notice that the choice of the $g_i$'s imply that $|Z(g_1):Z (G)| \ge p$ and 
$$|\prod_{i=1}^j Z(g_i):\prod_{i=1}^{j-1} Z(g_i)| \ge p$$ for each $j$ with $2 \le j \le t$.  This implies that $p^a = |A:Z(G)| \ge p^t.$  We conclude that $a \ge t \ge n/(b+l)$.
\end{proof}

\section{An application to Semi-extraspecial Groups}\label{ses}

A $p$-group $G$ is called {\it semi-extraspecial} if for all maximal subgroups $N$ of $Z(G)$, the quotient $G/N$ is an extraspecial group.  We will use s.e.s. groups to denote semi-extraspecial groups.  We believe that these groups were first studied by Beisiegel in \cite{beis}.  In our paper \cite{expos}, we give a detailed background of the research of these groups, and we do not wish to repeat that background.  We do want to mention that many of the results regarding these groups are due to Verardi and appear in \cite{ver}.

It is known that if $G$ is a s.e.s. $p$-group, then $|G:Z(G)|$ is a square, so there is a positive integer $n$ so that $|G:Z(G)| = p^{2n}$.  It is also known that if $m$ is the integer so that $|Z(G)| = p^m$, then $m \le n$.  If $G$ is a s.e.s. group where $n = m$, then $G$ is called {\it ultraspecial}.  It is also known that $G$ is special, i.e. that $G' = Z(G)$.

When $G$ is a s.e.s. group having the parameters of the previous paragraph, Verardi has shown that the largest possible order for a maximal abelian subgroup of $G$ is $p^{n+m}$ (see Theorem 1.8 of \cite{ver}).  It is not difficult to find for every prime $p$ and positive integers $m \le n$ s.e.s. groups $G$ with $|G:Z(G)| = p^{2n}$ and $|Z(G)| = p^m$ having maximal abelian subgroups of order $p^{n+m}$.  

In \cite{expos}, we present the results from the literature that show when $p$ is odd that the ultraspecial groups of order $p^{3n}$ having at least two maximal abelian subgroups of order $p^{2n}$ and exponent $p$ can be classified in terms of an algebraic object called a semifield.  Using these ideas, we show in \cite{sesone} when $p$ is odd how to construct all of the ultraspecial groups of order $p^{3n}$ and exponent $p$ having exactly one maximal abelian subgroup of order $p^{2n}$.  In \cite{genisom}, we show how to determine when these groups are isomorphic, and hence, we obtain a classification of these groups.  

On pages 148-149 of \cite{ver}, Verardi shows that there exist s.e.s. groups that do not have maximal abelian subgroups of this maximal possible order.  We and Josh Maglione have found many other examples of such groups.  The question arises of what can you say about the order of a maximal abelian subgroup in one of these groups.  

In addition, Verardi has shown that an ultraspecial group where all of the maximal abelian subgroups have this maximal possible order must be isoclinic to the Heisenberg group over a field (see Theorem 5.10 of \cite{ver}).  We say $G$ is the {\it Heisenberg group} of degree $p^a$ if $G$ is isomorphic to a Sylow $p$-subgroup of ${\rm GL}_3 (p^a)$.  Now, if when we consider s.e.s. groups that have an abelian subgroup of maximal possible order, but not all maximal abelian subgroups have this maximal possible order, it makes sense to ask what is the smallest possible order of maximal abelian subgroup.  




The next result is motivated by a question asked during a discussion with Thomas Keller during my visit to Texas State University in January 2018.  It is a corollary to Theorem \ref{threeia}.

\begin{corollary} \label{threei}
Let $G$ be an s.e.s. group with $|G:Z(G)| = p^{2n}$ and $|Z(G)| = p^m$.  Let $A$ is a maximal abelian subgroup of $G$ with $|A:Z(G)| = p^a$.
\begin{enumerate}
\item (General bound) Then $a \ge n/m$.  If $n \ge 2$, then $a \ge 2$.
\item If $p^l = {\rm max} \{|Z:Z(G)| \mid Z \in {\mathcal Z} (G) \}$, then $a \ge 2n/(m+l)$.
\end{enumerate}
\end{corollary}

\begin{proof}
We obtain (2) using Theorem \ref{threeia} by realizing that $|G:C_G (a)| = p^m$ for all $a \in G \setminus G'$ when $G$ is an s.e.s. group.  To obtain (1), we use the observation from Verardi that $l \le m$ (see Proposition 1.7 of \cite{ver}).  Thus, $m +l \le 2m$.  We have by (1), $a \ge 2n/(m+l) \ge 2n/2m = n/m$.  Note that if $a = 1$, then $|A:G'|$ would have prime order.  By Lemma \ref{cyclic} (2), we see that $A = C_G (g)$ for some element $g \in G$.  This implies that $|G:C_G (g)| = p^{2n - 1}$.  On the other hand, since $G$ is an s.e.s., we know that $|G:C_G (g)| \le p^n$, so we have a contradiction if $n \ge 2$.  Hence, we must have $a \ge 2$ if $n \ge 2$.
\end{proof}

Notice that if $G$ is an extraspecial group, then $m = 1$ in the notation of Corollary \ref{threei}, and conclusion (1) implies if $A$ is maximal abelian, then $|A:Z(G)| \ge p^{n/1} = p^n$.  Hence, we obtain as a consequence of Corollary \ref{threei} that all maximal abelian subgroups of an extraspecial group have the same size of $p^{n+1}$.  We expect that this is a known result, but we do not have a reference at this time.

On the other extreme, if $G$ is ultraspecial, then $m = n$.  From Corollary \ref{threei} we obtain that a maximal abelian subgroup $A$ satisfies $|A:G'| \ge p^2$.  It seems natural to ask if this is the best possible.  We now show that the answer is yes, this the best possible.

\begin{lemma}
Let $p$ be an odd prime, and let $n \ge 3$ be an integer.  Then there is an ultraspecial group of order $p^{3n}$ that has a maximal abelian subgroup of order $p^{n+2}$.
\end{lemma}`

\begin{proof}
Let $H$ be the $p$-group having exponent $p$ with the following generators: $a$, $b$, $c_1, \dots, c_{n-2}$, $z_1, \dots, z_n$.  We assume that each of the generators has order $p$.  We assume that the following commutators hold $[b,c_1] = z_1, \dots, [b,c_{n-2}] = z_{n-2}$, and other commutators either follow linearly from these or are trivial.  Note that $|H| = p^{2n}$, $Z (H) = \langle a, z_1, \dots, z_n \rangle$, and $H' = \langle z_1, \dots, z_{n-2} \rangle$.  We will define $Z = \langle z_1, \dots, z_n \rangle$.  It is not difficult to see that $C_H (b) = \langle b \rangle Z (H)$ and that $C_H (b)$ is a maximal abelian subgroup of $H$.  Also, $|C_H (b)| = p^{n+2}$.

Now, let $F$ be a semifield of order $p^n$.  (When $n = 3$, take $F$ to be the field of order $p^3$).  Identify $V$ with $F$, and take $\alpha: V \times V \rightarrow V$ to be the map obtained from the semi-field multiplication of $F$.   We can identify $H/Z$ and $Z$ with $V$, so that the commutator map in $H$, defines a bilinear map $\beta: V \times V \rightarrow V$ by $\beta (h_1Z,h_2Z) = \gamma[h_1,h_2]$ where $\gamma$ is the inverse for $2$ in $F$ (identified with $V$).  We use Theorem 3.1 of \cite{sesone} to construct the generalized semifield group $G = G(\alpha,\beta)$ whose associated semifield is $F$.  Observe that the subgroup identified as $B$ in Theorem 3.1 will be $C_G (a) = H$.  Notice that any abelian subgroup of $G$ containing $a$ must lie in $C_G (a)$.  Since $C_H (b)$ is a maximal abelian subgroup of $C_G (a)$ that contains $a$, it will be a maximal abelian subgroup of $G$.  This proves the result.
\end{proof}

The next question is whether one can find an ultraspecial group of order $p^{3n}$ where {\it all} of the maximal abelian subgroups have order $p^{n+2}$.  Our memory is that this is the specific question that Thomas had in mind in 2018, but now in February of 2024, we have again by asked this question by Rolf Brandl.  We note that if $G$ is an ultraspecial group of order $p^9$ which has no abelian subgroups of order $p^6$, then the exposition before the previous lemma says that every maximal abelian subgroup of $G$ will have order $p^5$.  (And we know such groups occur for $p = 3, 5, 7$.)  In fact, if $G$ is ultraspecial of order $p^9$ that is not the Heisenberg group, then $G$ will have maximal abelian subgroups of order $p^5$.  

Rolf then goes on to ask if ultraspecial groups of order $p^{3n}$ where all the maximal abelian subgroups have order $p^{n+2}$ must be isoclinic.  This question we can at least give a partial answer.  When $p = 3$ and $n = 3$, Josh Maglione has used Magma to compute all of the ultraspecial groups of order $p^9$ that do not have an abelian subgroup of order $p^6$ and they are isoclinic.  He also has been able to do this for $p = 5$ and $p=7$.  For $p= 5$, there are $2$ isoclinism classes and for $p=7$, there are $3$ isoclinism classes.  Hence, it seems likely for $p$ odd that there are $(p-1)/2$ isoclinism classes of ultraspecial groups of order $p^9$ that do not have an abelian subgroup of order $p^6$.  For $n> 3$, it seems likely that there will exist ultraspecial groups of order $p{3n}$ where all of the centralizers have order $p^{n+2}$, and since they are not isoclinic for $n=3$, it seems unlikely that they will be isoclinic for $n > 3$. 

We now continue with more speculation.  Notice that the question of the maximal abelian subgroups of an s.e.s. group $G$ is exactly the question of the maximal abelian subgroups of the centralizers of noncentral elements of $G$.  In particular, an abelian subgroup $A$ of $G$ is a maximal abelian subgroup of $G$ if and only if $A$ is a maximal abelian subgroup of $C_G (a)$ for some element $a \in A \setminus G'$.  (Notice that if $B \le C_G (a)$ is abelian for some element $a \in A \setminus Z(G)$ and $A \le B$, then $B \le C_G (b)$ for all elements $b \in A \setminus Z(G)$.  Thus, if $A$ is maximal abelian in $C_G (a)$ for some element $a \in A \setminus Z(G)$, then $A$ is maximal abelian in $C_G (a)$ for all elements $a \in A \setminus Z(G)$.)  A weaker question is: can we find an ultraspecial group of order $p^{3n}$ with an element $g \in G \setminus Z(G)$ so that all the maximal abelian subgroups in $C_G (g)$ have order $p^{n+2}$?  

For a prime $p$ and integer $n$, if we have a class $2$ group $H$ of order $p^{2n}$ and exponent $p$ so that $|Z(H)| > p^n$ and $|H'| \le p^n$, then we know that we can construct an ultraspecial group $G$ of order $p^{3n}$ that has $H$ as a centralizer of a nonidentity element of $G$.  In particular, if we can find such an $H$ so that $|Z(H)| = p^{n+1}$ and $C_H (h) = \langle h, Z(H) \rangle$ for all elements $h \in H \setminus Z(H)$, then we will have an ultraspecial group $G$ with a maximal abelian subgroup of order $p^{n+2}$.  Notice that this construction produces an abelian subgroup of order $p^{2n}$ in $G$.  When $n = 4$, it is not difficult to find examples.  (Take the generators $a_1, a_2, a_3, z$.  Assume $z$ is central.  Assume $[a_1,a_2]$, $[a_1,a_3]$, $[a_2,a_3]$ are distinct.)  (For $n = 5$, take generators $a_1,a_2,a_3, a_4, z$ where $z$ is central and assume that $[a_1,a_2]$, $[a_3,a_4]$, $[a_1,a_3]$, $[a_2,a_4]$, and $[a_1,a_4] = [a_2,a_3]$ are distinct.)  (Not sure if it can be done for $n = 6$.)

Conversely, if $G$ is ultraspecial of order $p^{3n}$ with all maximal abelian subgroup of order $p^{n+2}$, then there must exist an element $a$ whose centralizer $C = C_G (a)$ satisfies: $C' < Z(C)$, $|Z(C)| = p^{n+1}$, $|C'| \le p^n$ and $C_{C} (b) = \langle b,Z(a) \rangle $ for all $b \in C \setminus Z(a)$.  We would not be surprised if such a group does not exist for large enough $n$.

Before we close the paper, we have another question and observation.  The question is: If $G$ is an s.e.s. group and $a, b_1, \dots, b_n$ are elements of $G \setminus G'$, then are the following true: 
\begin{enumerate}
	\item  either $Z (a) \le \prod_{i=1}^n Z (b_i)$ or $Z (a) \cap \prod_{i=1}^n Z (b_i) = G'$; and 
	\item either $Z (a) \le Z (\cap_{i=1}^n C_G (b_i))$ or $Z (a) \cap Z (\cap_{i=1}^n C_G (b_i)) = G'$?
\end{enumerate}
The observation is: if (1) is true, then we can prove the following using the proof of Corollary \ref{threei}:  let $p^k = {\rm min} \{ |Z:Z(G)| \mid Z \in {\mathcal Z} (G), |Z:Z(G)| < p^n \}$.  If $A$ is maximal abelian in $G$ and $|A:Z(G)| < p^n$, then $|A:Z(G)| \ge p^{2kn/(m+l)}$.


\end{document}